\documentclass[11pt]{amsart}
\usepackage{mathtools, amssymb, amsfonts, amsthm}
\usepackage{epsfig, psfrag, color, multicol, graphicx}
\usepackage{amsfonts}
\usepackage{epstopdf}

\newtheorem{thm}{Theorem}[section]

\newtheorem{prop}[thm]{Proposition}

\theoremstyle{plain}
\newtheorem{theo}[thm]{Theorem}
\newtheorem{lem}[thm]{Lemma}

\theoremstyle{definition}

\newtheorem{rem}[thm]{Remark}

\numberwithin{equation}{section}

\def\zz{\mathbb Z}
\def\nn{\mathbb N}

\def\rr{\mathbb R}

\def\ov{\overline}

\def\al{\alpha}
\def\be{\beta}

\def\T{\mathbf{T}}

\def\ssu{\subset}

\def\wt{\widetilde}
\def\<{\langle}
\def\>{\rangle}

\def\Z{ {\text {\rm Z} } }

\def\Q{{\text {\rm Q} } }

\def\0{{\mathbf 0}}

\def\.{\hskip.06cm}
\def\ts{\hskip.03cm}
\def\mts{\hspace{-.04cm}}

\def\conv{{\text {\rm {conv}} }}

\def\bz{{\textbf{z}}}
\def\bx{{\textbf{x}}}
\def\bt{{\textbf{t}}}
\def\by{{\textbf{y}}}

\def\poly{\textup{\textsf{P}}}
\def\P{{\textup{\textsf{P}}}}
\def\FP{{\textup{\textsf{FP}}}}

\def\SigmaP{\Sigma^{\poly}}
\def\PiP{\Pi^{\poly}}

\def\Z{\mathbb{Z}}
\def\R{\mathbb{R}}

\def\Q{\mathbb{Q}}

\def\P{\mathcal{P}}
\def\PF{\mathcal{PF}}
\def\T{\mathcal{T}}
\def\F{\mathcal{F}}

\newcommand{\cj}[1]{\overline{#1}}
\newcommand{\n}{\cj{n}}
\renewcommand{\b}{\cj{b}}
\def\albar{\cj{\alpha}}
\def\nubar{\cj{\nu}}
\def\a{\cj{a}}

\newcommand{\x}{\mathbf{x}}
\renewcommand{\t}{\mathbf{t}}
\renewcommand{\u}{\mathbf{u}}
\def\v{\mathbf{v}}
\def\zzeta{\boldsymbol{\zeta}}
\def\xxi{\boldsymbol{\xi}}

\newcommand{\y}{\mathbf{y}}
\newcommand{\z}{\mathbf{z}}

\newcommand{\floor}[1]{\lfloor#1\rfloor}

\newcommand{\polyin}[1]{\textup{poly}(#1)}
\newcommand{\Pclass}{\P_{k,\n,a}}
\newcommand{\Pbclass}{\P^{\textup{b}}_{k,\n,a}}
\newcommand{\PFclass}{\PF_{k,\n,a}}

\newcommand{\ex}{\exists}
\renewcommand{\for}{\forall}
\newcommand{\Qb}{Q^{\textup{b}}}
\newcommand{\exb}{\exists^{\textup{b}}}
\newcommand{\forb}{\forall^{\textup{b}}}
\newcommand{\Qbk}[1]{\Qb_{#1} \x_{#1}}

\def\gf{\text{GF} }
\def\tauHad{\, \star_{\tau} \,}

\def\nin{\noindent}

\def\bounded{\textup{b}}

\def\proj{\textup{proj}}

\def\NP{{\textup{\textsf{NP}}}}

\title{Complexity of short Presburger arithmetic}

\author[Danny Nguyen \and Igor Pak]{Danny Nguyen$^{\star}$ \and Igor~Pak$^{\star}$}

\thanks{\thinspace ${\hspace{-.45ex}}^\star$Department of Mathematics,
UCLA, Los Angeles, CA, 90095.
\hskip.06cm
Email:
\hskip.06cm
\texttt{\{ldnguyen,\ts{pak}\}@math.ucla.edu}}

\thanks{
\today}

\begin{document}
\maketitle


\begin{abstract}
We study complexity of short sentences in Presburger arithmetic (\textsc{Short-PA}).
Here by ``short''  we mean sentences with a bounded number of variables, quantifiers,
inequalities and Boolean operations; the input consists only of the integers involved
in the inequalities.  We prove that assuming Kannan's partition can be found in
polynomial time, the satisfiability of \textsc{Short-PA} sentences can be decided
in polynomial time.  Furthermore, under the same assumption, we show that
the numbers of satisfying assignments of short Presburger sentences
can also be computed in polynomial time.
\end{abstract}

\vskip1.5cm

\section{Introduction}

\subsection{The results}
We consider \emph{short Presburger sentences} defined as follows:
$$
(\ast) \qquad
\ex \ts\x_{1} \;\forall \ts\x_{2} \; \ex \ts\x_{3} \;\dots \; \forall/\exists \ts \x_{k} \,:\, \Phi\bigl(\x_{1}, \dots, \x_{k}\bigr),
$$
where the quantifiers alternate, the variables $\x_i \in \zz^{n_i}$ have fixed dimensions $\ov n=(n_1,\ldots,n_k)$, and $\Phi(\x_1,\ldots,\x_k)$ is a fixed Boolean combination of linear systems of the form:
$$(\ast\ast) \qquad
A_1\ts \x_1 \. + \. \ldots \. + \. A_k\ts \x_k \, \le \, \ov b\ts.
$$
In other words, everything is fixed in~$(\ast)$ except for the entries of the matrices~$A_i$ and of the vectors~$\ov b$ in $(\ast\ast)$.

Let \textsc{Short-PA} be the satisfiability problem of sentences~$(\ast)$.
This is one of the few remaining gaps in complexity of the first order
logic problems.  If any of the conditions are weakened (unbounded quantifiers,
variables,  or linear systems), the problem becomes
\NP-complete or even super-exponential (see below).

The \textsc{Short-PA} generalizes \emph{Integer Linear Programming} in
fixed dimension (cf.~$\S$\ref{LenstraKannanShort}), which can be
viewed as satisfiability of sentences
$$(\circ) \qquad
\exists \ts\bx \, : \, A\ts\bx \, \le \, \ov b\ts
$$
with $\x \in \zz^{n}$ for a fixed $n$.  Satisfiability of~$(\circ)$ in polynomial
time is due to Lenstra~\cite{L}.  Its proof relies on difficult results
in geometry of numbers (see the discussion below).

Similarly, \textsc{Short-PA} generalizes \emph{Parametric Integer Linear Programming}
in fixed dimension (cf.~$\S$\ref{LenstraKannanShort}), which can be
viewed as satisfiability of sentences
$$(\circ\circ) \qquad
\forall \ts\by \in Q \;\; \exists \ts\bx \, : \, A\ts\bx \. + \. B\ts\by \, \le \, \ov b\ts
$$
with $\x \in \zz^{n}$ and $\y \in \zz^{m}$ for fixed $n$ and $m$.
Here $Q$ is another rational polyhedron, described by another system $C\ts \y \le \cj d$.

Satisfiability of~$(\circ\circ)$ in polynomial time is due to Kannan~\cite{K1}
(Theorem~\ref{t:Kannan}).  His proof crucially relies on \emph{Kannan's partition theorem} (KPT)
(Theorem~\ref{KannanPartition}), which is somewhat technical and can be described as
follows.  KPT says that there is a partitioning of $\zz^m$ into a polynomially many
polyhedral regions~$P_i$, $1\le i \le r$, such that in order to solve for an
$\x \in \Z^{n}$ satisfying $A\x \le \b$  with $\b$ changing, one only need to
preprocess the matrix $A$ in polynomial time, and
from there get the regions~$P_i$. Then, when queried with
$\b \in P_i$, one only need to check for a finite number ($n^{4n}$) of candidate
solutions~$\x\in \zz^n$, which are called \emph{test points}.

In this paper we repeatedly use KPT as a black box, to prove the following general result:

\medskip

\nin
{\bf Theorem~A.} \ts \emph{Assuming KPT, problem \textsc{Short-PA} is in \poly. }

\medskip

The proof of our Theorem~A uses quantifier elimination inductively, with each
inductive step applying KPT in the case $m=1$.

\medskip

Let us emphasize that even the following special case of~$(\ast)$ remained wide open:
$$(\circ\mts\mts{}\circ\mts\mts{}\circ) \qquad
\exists \ts\bz \in R \;\; \forall \ts\by \in Q \;\; \exists \ts\bx \, : \, A\ts\bx \. + \. B\ts\by \. + \. C\ts\bz \, \le \, \ov b\ts.
$$
This case was singled out by Kannan in~\cite{K2} as the next challenge.

There is a natural geometric way to view these problems.  Problem~$(\circ)$
asks whether a given rational polyhedron $P \ssu \rr^d$ contains an integer point.
Problem~$(\circ\circ)$ asks whether the projection of~$P$ contains all integer points in some polyhedron~$Q$.
Finally, problem~$(\circ\mts\mts{}\circ\mts\mts{}\circ)$ asks whether there is an $R$-slice of a
polyhedron $P$ for which the projection contains all integer points in some polyhedron~$Q$.

Note that in the above three problems, the restriction in each quantifier can be pushed inward at the cost of introducing extra Boolean operators. For example:
\begin{equation*}
 \forall \ts\by \in Q \;\; \exists \ts\bx \, : \, A\ts\bx \. + \.  B\ts\by \, \le \, \ov b \quad  \iff  \quad \forall \ts\by \; \exists \ts\bx \, : \, (\y \notin Q) \lor (A\ts\bx \. + \. B\ts\by \, \le \, \ov b).
\end{equation*}

\medskip

Our next result is a counting analogue of Theorem~A.  By analogy with~$(\ast)$,
define a \emph{short Presburger formula} as a set of the form:
$$
(\ast') \qquad
\bigl\{\ts \x_{1} \, : \, \ex\ts\x_{2} \; \for\ts\x_{3} \;\dots \; \ex/\for\ts\x_{k} \;\, \Phi\bigl(\x_{1}, \x_{2}, \dots, \x_{k}\bigr)\ts\bigr\}\ts,
$$
where the dimensions and the Boolean combinations are fixed as in $(*)$.
Let \textsc{$\#$Short-PA} be the counting problem of the number of satisfying
assignments $\x_{1}$ of a short Presburger formula $(\ast')$.
The complexity of \textsc{$\#$Short-PA} was stated as an open problem
by Barvinok~\cite[$\S$5]{B2}, and as a conjecture by Woods~\cite{Woods}
(see also~\cite{W}).

\medskip

\nin
{\bf Theorem~B.} \ts \emph{Assuming KPT, the counting problem} $\#$Short-PA \. \emph{is in~\FP. }

\medskip

This is an extension of Theorem~A, as counting easily implies decision.
Following an example above, a special case of Theorem~B
computes the number of integer points defined
in~$(\circ\mts\mts{}\circ\mts\mts{}\circ)$.
The proof of Theorem~B is inductive and again uses KPT for reduction of the
number of quantifiers.  We use the Barvinok--Woods theorem (Theorem~\ref{BW})
as a base of induction.


\subsection{Historical overview}
Presburger arithmetic was introduced by Presburger in~\cite{Pre},
where he proved it is a decidable theory.
The general theory allows unbounded numbers of quantifiers, variables and Boolean operations.
 A quantifier elimination
(deterministic) algorithm was given by Cooper~\cite{C}, and was
shown to be triply exponential by Oppen~\cite{O} (see also~\cite{RL}).
A nondeterministic doubly exponential complexity lower bound was obtained
by Fischer and Rabin~\cite{FR} for the general theory.  This pioneering result was further
refined to
simply exponential nondeterministic lower bound for a bounded
number of quantifier alternations~\cite{Fur} (see also~\cite{Sca}).
Of course, in all these cases the number of variables is unbounded.

In~\cite{S}, Sch\"{o}ning proves \NP-completeness for two quantifiers
$\ts\exists x \ts \forall y\ts:\ts\Phi(x,y)$, where $x,y\in \zz$
and $\ts\Phi(x,y)\ts$ is a quantifier-free Presburger expression.
Here the expression $\ts\Phi(x,y)\ts$ has an unbounded number
of inequalities and Boolean combinations.
This improved on an earlier result by \cite{G}, who also established
that similar sentences with $k+1$ quantifier alternations and a bounded
number of variables are complete for the $k$-th level in the Polynomial Hierarchy.

In a positive direction, the progress has been slow.  The first
breakthrough was made by Lenstra~\cite{L}  (see also~\cite{Schrijver}),
who showed that the \emph{integer feasibility problem} $(\circ)$ can be solved
in polynomial time in a fixed dimension (see also~\cite{E1,FT} for better bounds).
The next breakthrough was made by Kannan~\cite{K1} (see also~\cite{K2}),
who showed how to solve \emph{parametric integer linear programs} $(\circ\circ)$ in fixed dimensions.
This result was further strengthened in~\cite{ES} (see also~\cite{E2}).
All of these greatly contrast with the hardness results from~\cite{S}
and~\cite{G}, because here only conjunctions of inequalities are allowed.

Barvinok~\cite{B1} showed that integer points in a convex polytope $P\ssu\rr^d$ can be counted in polynomial time,
for a fixed dimension~$d$.  He utilized the \emph{short generating function} approach pioneered
by Brion, Vergne and others (see~\cite{B3} for details and references).
Barvinok and Pommersheim~\cite{BP} extend this approach to prove
that integers points in a Boolean combination of polytopes can also
be counted in polynomial time.  This is in contrast with~\cite{EH}, which
proves that minimizing the number of integer points $\x$ satisfying $(\circ)$ over
different~$\ts\ov b\ts$ is $\NP$-hard.   Barvinok and Woods showed how to
count integer points in projections of (single) polytopes in polynomial
time~\cite{BW}.  Woods~\cite{W} also showed that Presburger formulas
can be characterized by having rational generating functions (see also~\cite{Woods}).
Theorem~B can be viewed as algorithmic version of this result, when the formula is short.

Barvinok's algorithm has been simplified and improved in~\cite{DK,KV};
it was also extended to various
integral sums and valuations over convex polyhedra~\cite{B+,B3,BV}.
The algorithm has important applications in a number of areas, ranging from
polynomial optimization \cite{DHKW1,DHKW2} to representation theory~\cite{CDW,PP},
to commutative algebra \cite{D+,MS} and to random sampling~\cite{Pak}.
Both Barvinok's and Barvinok--Woods' algorithms have been
implemented and used for practical computation \cite{DHTY,Kop,VSBLB}.

\subsection{Proof features and previous obstacles}
The proofs of theorems~A and~B have some unusual features when compared to
other recent work in the area.  First, we use a quantifier elimination
technique in the classical style of the formal arithmetic theory.
However, we treat Boolean formulas geometrically, in the style
of Barvinok et al., to allow the applications of KPT.  Let us emphasize that having Boolean formulas
is crucial for our proof -- without them the inductive argument crumbles, even for
sentences like $(\circ\mts\mts\mts\circ\mts\mts\mts\circ)$ above.  We refer
to~$\S$\ref{LenstraKannanShort} for a related phenomenon.

Second, the proof of Theorem~B crucially relies on the
technology of \emph{short generating functions} (GF)
$$(\divideontimes) \qquad f(\bt) \, = \,
\sum_{i=1}^N \, \frac{c_{i} \. \bt^{\a_i}}{(1-\bt^{\b_{i\ts 1}})\cdots (1-\bt^{\b_{i\ts k_{i}}})}\ts,
$$
where $c_{i} \in \mathbb{Q}, \; \cj a_{i}, \cj b_{ij} \in \zz^{n}$ and $\ts\bt^{\a}$ denotes $t_1^{a_{1}}\cdots {}\ts t_n^{a_{n}}\ts$ for $\ts\a = (a_{1},\ldots,a_{n}) \in \zz^{n}$.
We caution the reader that word ``short'' in ``short GF'' only means that the GF
is given in the form ($\divideontimes$).
It does not necessarily  mean the GF has polynomial size.
As we mentioned earlier, short GFs are a wonderful tool which allows one to take finite
unions, intersections, complements and substitutions.
Unfortunately, there is no easy way to take projections on the level of
short~GFs; the hardness result was recently proved in~\cite{W} (see also~\cite{NP2}).

The reader can be understandably confused at this point since the ability to take
projections is exactly the statement of the Barvinok--Woods theorem.  The problem is
quite delicate here: having switched from polytopes to short GFs, the Barvinok--Woods
technique cannot be iterated.  Here is a simple way to think about it.  The Barvinok--Woods theorem
allows one to efficiently compute short GFs for projections of (single) polytopes $P_1,\ldots,P_r$ in polynomial time.
Call these projections $\proj(P_{1}),\dots,\proj(P_{r})$.
Earlier tools by Barvinok and Pommersheim also allow one to compute a short GF for the union
$Y=\proj(P_1)\cup \ldots \cup \proj(P_r)$ when $r$ is bounded.
However, now that the polytopal structure is lost, there is no easy way to compute in polynomial
time another projection of~$Y$ when we are given only a short GF for~$Y$.  In fact, we recently
prove that this is computationally hard in~\cite{NP2}.

\bigskip

\section{Notations}

\nin
We use $\nn \ts = \ts \{0,1,2,\ldots\}$.

\nin
Unspecified quantifiers are denoted by $Q_{1}, Q_{2}$, etc.

\nin
Unbounded (unrestricted) quantifiers are denoted $\for$ and $\ex$.

\nin
Bounded (restricted) quantifiers are denoted $\forb$ and $\exb$.

\nin
Unquantified Presburger expressions are denoted by $\Phi, \Psi, \Gamma$, etc.

\nin
We use $\Lambda$ to denote a linear system.

\nin
We use $\left[
\begin{smallmatrix}
a \\
b
\end{smallmatrix}
\right]
$ to denote a disjunction $(a \lor b)$ and $
\big\{
\begin{smallmatrix}
a \\
b
\end{smallmatrix}
\big\}
$ to denote a conjunction $(a \land b)$.

\nin
All constant vectors are denoted $\n, \b, \albar, \nubar$, etc.

\nin
We use $0$ to denote both zero and the zero vector.

\nin
The $L_{1}$ norm of a vector $\n$ is denoted by $|\n|$.

\nin
All matrices are denoted $A, B$, etc.

\nin
All integer variables are denoted $x,y,z$, etc.

\nin
All vectors of integer variables are denoted $\x, \y, \z$, etc.

\nin
If $x_{j} \le y_{j}$ for every index $j$ in vectors $\x$ and $\y$, we write $\x \le \y$.

\nin
If $x_{j} \le c$ for every index $j$ with $c$ a constant, we write $\x \le c$.

\nin
We use $\floor{.}$ to denote the floor function.

\nin
The the vector $\y$ with coordinates $y_{i} = \floor{x_{i}}$ is denoted by $\y = \floor{\x}$.

\nin
GF is an abbreviation for ``\emph{generating function}''.

\nin
Single-variable GFs are denoted by $f(t), g(u), h(v)$, etc.

\nin
Multi-variable GFs are denoted by $A(\t), B(\u), a(\v)$, etc.

\nin
The function $\phi(\cdot)$ denotes the (binary) length of a formula, GF, matrix, vector, etc.

\nin
Half-open intervals are denoted by $[\al,\be)$, etc.

\nin
A \emph{polyhedron} is an intersection of finitely many closed half-spaces in some euclidean space~$\rr^n$.

\nin
A \emph{copolyhedron} is a polyhedron with possibly some open facets.

\nin
A \emph{polytope} is a bounded polyhedron.

\bigskip

\section{Short Presburger sentences}

\subsection{Deciding short Presburger sentences}

We consider a fixed class of short Presburger sentences in prenex normal form

\begin{equation}\label{PclassDef}
\P_{k,\n,a} = \Big\{ S = \big[ Q_1 \x_{1} \; Q_2 \x_{2} \; \dots \exists \x_{k} \,:\, \Phi(\x_{1}, \dots, \x_{k}) \big] \Big\}.
\end{equation}
{}

\noindent Here $Q_1,\dots,Q_{k} \in \{\for, \ex\}$ are $k$ alternating quantifiers with $Q_{k} = \ex$, each $\x_{i} \in \Z^{n_{i}}$ with fixed dimensions $\n=(n_{1}, \dots, n_{k})$, and $\Phi$ is a Boolean combination of at most $a$ rational inequalities in $\x_i$'s.
We can also assume each $\x_{i} \ge 0$, because every integer variable can be represented as the difference between $2$ nonnegative variables, and doing so only increases each $n_i$ by a factor of $2$.
For a sentence $S \in \P_{k,\n,a}$, we denote by $\phi(S)$ the binary length of~$S$. Now Theorem~A can be restated as follows:

\begin{theo}\label{th:main_1}
Assuming KPT, every $S \in \Pclass$ can be decided in polynomial time
with respect to~$\phi(S)$.  The polynomial degree depends only on $k,\n$ and~$a$.
In other words, $\Pclass \in \poly$ for every $k,\n,a$.
\end{theo}

As we mentioned in the introduction, from Kannan's Theorem 3.2 in \cite{K1}, every such class $\Pclass$ with $k = 2$ can be decided in polynomial time with respect to $\phi(S)$, with the polynomial degree depending on $\n$ and $a$. In the literature, the case $k = 2$ is called Parametric Integer Linear Programming, because every such problem has the form $\forall\y \; \exists\x : \Phi(\y,\x)$, where $\y$ varies over the parameter space $\Z^{n_1}$, and for each such $\y$ we need to solve an Integer Linear Programming problem for $\x \in \Z^{n_2}$.

\begin{prop}\label{p:Sigma}
$\Pclass \in \SigmaP_{k-2}$ if $k$ is odd and $\Pclass \in \PiP_{k-2}$ if $k$ is even.
\end{prop}


\begin{proof}[Proof of Proposition~\ref{p:Sigma}]
From a general result in \cite{G}, we know $\Pclass \in \SigmaP_{k} / \PiP_{k}$ when $k$ is odd/even because there are only a bounded number of quantified variables.
In other words, this says that for every $S \in \Pclass$, it suffices to verify $S$ for all $\x_{i}$ with coordinates $x_{i,j}$ less than $2^{\ell_{i}}$.
Here $\ell_{1},\dots,\ell_{k}$ are polynomial in $\phi(S)$ and can also be computed in polynomial time from $S$. Furthermore, given $(\x_{1},\dots,\x_{k-2})$, Theorem~\ref{t:Kannan} allows us to check whether $\for \x_{k-1}  \ex \x_{k} : \Phi(\x_{1},\dots,\x_{k})$ in polynomial time. Therefore, we get $\Pclass \in \SigmaP_{k-2} / \PiP_{k-2}$ if $k$ is odd/even.
\end{proof}


By the above proposition, to decide a statement $S \in \Pclass$, it is enough restrict the coordinates $x_{ij}$ in $\x_{i}$ to an interval $[0,2^{\ell_{i}})$.
Here $\ell_{1},\dots,{\ell_{k}}$ are polynomial in $\phi(S)$ and also computable in polynomial time given $S$.
We can change each quantifier $\ts Q_i \x_i \ts$ to $\ts \Qb_i \x_i \ts$, where the superscript ``b'' means that $\forall$/$\exists\ts \x_i \in [0,2^{\ell_{i}})^{n_{i}}$.
Thus, we can recast each class $\Pclass$ as consisting of polynomial size search problems:

\begin{equation}\label{Pbclass}
\P^{\bounded}_{k,\n,a} = \Big\{ S = \big[ \Qb_1 \x_{1} \; \Qb_2 \x_{2} \; \dots \; \exb \x_{k} \,:\, \Phi(\x_{1}, \dots, \x_{k}) \big] \Big\}.
\end{equation}
{}

\begin{lem}\label{red1}
For a sentence $S \in \Pbclass$ as in~\eqref{Pbclass}, we can convert~$\Phi$
to a short system $($conjunction$)$ of inequalities at the cost of increasing
the length~$\phi(S)$ by a polynomial factor, and increasing $n_{k}$ and $a$ by some constants.
\end{lem}

\begin{proof}
First let $n = n_{1} + \ldots + n_{k}$ and $\x = (\x_{1},\dots,\x_{k}) \in \zz^{n}$, we can rewrite $\Phi(\x_{1}, \dots, \x_{k})$ as a DNF:
\begin{equation}\label{DNF}
\Phi(\x_{1},\dots,\x_{k}) = (A_1 \x \le \b_{1}) \lor \dots \lor (A_{t} \x \le \b_{t}).
\end{equation}
Here each short system $A_{j} \x \le \b_{j}$ contains at most~$a$ inequalities and defines a polytope $P_j \subset \R^{n}$ (because each $\x_i$ is bounded). The the total number $t$ of such systems is also at most~$2^{a}$. So $\Phi$ defines a union of $t$ polytopes (intersecting $\Z^{n}$). We claim that there exists a polytope $R \subset \R^{m}$ with $m = t+n$ so that for every $\x \in \zz^{n}$, we have:
\begin{equation}\label{proj}
\x \in \bigcup_{i=1}^{t} P_j  \quad \iff \quad \ex \t \in \Z^t : (\t,\x) \in R.
\end{equation}
To see this, we first define
\[
R_{j} = (0,\dots,0,1_{j},0,\dots,0,P_{j}) \subset \R^{m}.
\]
Explicitly, each $R_{j}$ is $P_{j}$ augmented with $t-1$ coordinates $0$, and a coordinate $1$ in the $j$-th position.
Now we can define
\begin{equation}\label{convexhull}
R = \text{conv} \{ R_{1} , R_{2} , \dots , R_{t} \}.
\end{equation}
It is easy to see that every integer point $(\t,\x)$ in $R$ must be in some $R_{j}$, and vice versa. This establishes~\eqref{proj}.

The vertices of each $P_{j}$ can be computed in polynomial time from its facets.
The vertices of $R_{j}$ come directly from those of $P_{j}$.
The vertices of $R$ are all vertices of $R_{j}$ for $1 \le j \le t$.
The facets of $R$ can be computed in polynomial time from its vertices because the total dimension $m=n+t$ is bounded. So the polytope $R$ can be presented as
\[
A(\t,\x) \le \b
\]
with both $A$ and $\b$ computable in polynomial time.
The original sentence $S$ can now be written in an equivalent form:
\begin{equation}\label{onesystem}
\Qb_1 \x_{1} \; \Qb_2 \x_{2} \; \dots \; \forb \x_{k-1} \; \exb \wt \x_{k}  \,:\, A \wt \x  \le \b,
\end{equation}
where $\wt\x_{k} = (\x_{k},\t)$ and $\wt \x = (\x, \t)$.
By merging $\exb \x_{k}$ and $\exb \t$ to form $\exb \wt \x_{k}$, we get $n_{k} \gets n_{k} + t \le n_{k} + 2^{a}$.

Note that the system $A \wt \x \le \b$ is still short. This can be seen as follows. Each system in \eqref{DNF} contains at most~$a$ inequalities, so each $P_{j}$ has at most $a^{n}$ vertices.
Each $R_{j}$ has the same number of vertices as $P_{j}$. Thus, the polytope $R$ in \eqref{convexhull} has at most $t a^{n} \le 2^{a} a^{n}$ vertices. Therefore, the number of facets of $R \subset \R^{m}$ is at most
\[
\left( 2^{a} a^{n} \right)^{m} \le \left( 2^{a} a^{n} \right)^{n+2^{a}},
\]
which is a constant. Each facet of $R$ can be computed in polynomial time, so it also has a polynomial length description.

We conclude that both $n_{k}$ and $a$ are changed by contants depending only on $\n, a$ and~$k$. The new system of inequalities is short, and has length bounded by a polynomial factor.
\end{proof}

\begin{rem}\label{r:extra-dim}
The extra dimension for $\t$ in the above proof can actually be lowered to $a$.
Recall that there are at most $2^{a}$ polytopes $P_{i}$.
We can pick $2^{a}$ points $\cj{r}_{1},\dots,\cj{r}_{2^{a}} \in \{0,1\}^{a}$ and define
\[
R_{j} = (\cj{r}_{j}, P_{j}) \subset \R^{m},
\]
where $m$ is now $a+n$.
Notice that $\cj{r}_{1},\dots,\cj{r}_{2^{a}}$ are vertices of the $a$-dimensional unit cube, which has no interior integer points.
Therefore, the convex hull $R = \conv(R_{1},\dots,R_{t})$ still satisfies the property
\[
\y \in R \cap \Z^{m} \quad \iff \quad \y \in R_{j} \cap \Z^{m} \text { for some $j$}.
\]
\end{rem}

By the above lemma, at the cost of a polynomial factor, we can restrict our attention to the subclass of $\Pbclass$ for which the $\Phi$ is just a short system of inequalities.

\begin{lem}\label{red2}
Every short sentence $S \in \Pbclass$ of the form
$$
\Qbk{1} \; \Qbk{2} \; \dots \; \forb \x_{k-1} \; \exb \x_{k}  \,:\,  \Phi(\x_1, \dots, \x_k)
$$
is equivalent to a short sentence $S'$ of the form
\begin{equation}\label{red2_eq}
\Qb_1 y_1 \; \Qb_2 y_2 \; \dots \; \forb y_{k-1} \; \exb \y_k \,:\, \Psi(y_1, \dots, y_{k-1}, \y_{k}),
\end{equation}
where $y_1, \dots, y_{k-1}$ are singletons, $\y_k \in \Z^{m}$ with $m \le n_1 + \ldots + n_k$, and $\Psi$ is a short system of length polynomial in $\phi(S)$ that describes a polytope in $\R^{m+k-1}$.
\end{lem}

\begin{proof}
Since all quantifiers are bounded, we can assume $0 \le x_{i,j} < 2^{\ell_{i}}$ for all coordinates $x_{i,j}$ in $\x_{i}$, where $1 \le i \le k, 1 \le j \le n_{i}$.
Therefore, we can uniquely represent each vector $\x_{i}$ by a single integer $y_{i}$, where
\[
y_{i} = x_{i,1} + 2^{\ell_{i}}x_{i,2} + \ldots + 2^{(n_{i}-1)\ell_{i}}x_{i,n_{i}}.
\]
Now each variable $y_{i}$ is bounded in the range $[0,2^{n_{i} \ell_{i}})$, and we can replace $\x_{i}$ by $y_{i}$ for all $1 \le i \le k-1$. However, in order to recover all the coordinates $x_{i,j}$ in the system~$\Phi$, we need to augment $\x_{k}$ by $(n_{1} + \ldots + n_{k-1})$ extra coordinates. So let $\y_{k} = (y_{k,1} \; , \dots, \; y_{k,m})$, where $m = n_{1} + \ldots + n_{k}$. We identify the last $n_{k}$ coordinates in $\y_{k}$ with those of $\x_{k}$. For the first $m-n_{k}$ coordinates of $\y_{k}$, we condition
\[
\begin{Bmatrix*}[l]
y_{1} &= &y_{k,1}  +  2^{\ell_{1}} y_{k,2} + \ldots + 2^{(n_{1} - 1) \ell_{1}} y_{k,n_{1}} \\
~~~~\\
y_{2} &= &y_{k,n_{1}+1} +  2^{\ell_{2}} y_{k,n_{1}+2} +\ldots + 2^{(n_{2} - 1) \ell_{2}} y_{k,n_{1}+n_{2}} \\
~ &\, \ts\vdots &~ \\
y_{k-1} &= &y_{k,n_{1} + \ldots + n_{k-2} +1} + \ldots + 2^{(n_{k-1} - 1) \ell_{k-1}} y_{k,n_{1}+\ldots+n_{k-1}} 
\end{Bmatrix*}.
\]
Besides, we require $0 \le y_{k,j} < 2^{\ell_{i}}$ for each $y_{k,j}$ in the $i$th row of the above system.
Adding all the above conditions (as linear inequalities) into the new system $\Phi$, where each variable $x_{i,j}$
is substituted by $y_{k, n_{1} + \ldots + n_{i-1} + j}$, we obtain an equivalent
short system $\Psi(y_{1}, \dots, y_{k-1}, \y_{k})$ of length $\text{poly}(\phi(S))$.
\end{proof}


Next, we disassociate $y_1,\dots,y_{k-2}$ from $\Psi(y_1, \dots, y_{k-1}, \y_{k})$ to obtain a system $\Lambda(y_{k-1}, \y_{k})$ in only the last two variables $y_{k-1}$ and $\y_{k}$. The following lemma shows this can be done at a cost of introducing extra relations $R_{1}(y_1,y_2), \dots, R_{k-2}(y_{k-2},y_{k-1})$, which are all short.

\begin{lem}\label{red3}
Every short sentence $S'$ of the form
\begin{equation*}
\Qb_1 y_1 \; \Qb_2 y_2 \; \dots \; \forb y_{k-1} \; \exb \y_k : \Psi(y_1, \dots, y_{k-1}, \y_{k})
\end{equation*}
is equivalent to another short sentence $S''$ of the form
\begin{equation}\label{type_1}
\aligned
\exb z_1 \; & \forb z_2 \; \lnot  R_1(z_1, z_2) \lor \Big[ \exb z_3 \; R_2(z_2,z_3) \land \bigl[ \dots \\ & \dots \lnot R_{k-2}(z_{k-2},z_{k-1}) \lor [ \exb \z_k \; \Lambda(z_{k-1}, \z_{k})  ] \dots \bigr] \;  \Big]
\endaligned
\end{equation}
if $k$ is odd, i.e., $\Qb_k = \exb$, or
\begin{equation}\label{type_2}
\aligned
\forb z_1 \; & \exb  z_2  \;  R_1(z_1, z_2)  \land \Big[ \forb z_3 \; \lnot R_2(z_2,z_3) \lor  \bigl[ \dots \\ & \dots \lnot R_{k-2}(z_{k-2},z_{k-1}) \lor [ \exb \z_k \; \Lambda(z_{k-1}, \z_{k})  ] \dots \bigr] \;  \Big]
\endaligned
\end{equation}
if $k$ is even, i.e., $\Qb_k = \forb$.


\noindent Here $R_{1}, \dots, R_{k-2}$ and $\Lambda$ are all short and quantifier free.
Also $\Lambda$ is a short system of inequalities with length $\polyin{\phi(S')}$.
\end{lem}

\begin{proof}
By the bounded quantifiers, we have $y_{i} \in [0,2^{\ell_{i}})$ for $1 \le i \le k-1$ and $y_{k,j} \in [0,2^{\ell_{k}})$ for $1 \le j \le n_{k}$. We will make new variables $z_{1}, \dots, z_{k-1}$ and condition them so that each $z_{i}$ express $\. y_{1},\dots,y_{i} \.$ concatenated in binary. We identify $z_{1}$ with $y_{1}$.
For $z_{2}$, we concatenate $y_{1}$ and $y_{2}$.
This just means that $z_{2}$ has $\ell_{1}+\ell_{2}$ binary digits, with the first (most significant) $\ell_{1}$ digits from~$y_{1}$ (now $z_{1}$), and the last (least significant) $\ell_{2}$ digits from $y_{2}$.
In other words, we have $z_{1} = \floor{z_{2} / 2^{\ell_{2}}}$.
So the first condition $R_{1}(z_{1},z_{2})$ is:
\[
R_{1}(z_{1}, z_{2}) : z_{1} = \floor{z_{2} / 2^{\ell_{2}}} \quad \iff \quad
\bigg\{
\begin{matrix*}[l]
z_{1} \le z_{2} / 2^{\ell_{2}} \\
z_{1} > z_{2} / 2^{\ell_{2}} - 1
\end{matrix*}
\;\;.
\]
In general, if $t_{j} = \ell_{1} + \dots + \ell_{j}$, then for any $1 \le j \le k-2$, the variable $z_{j+1}$ has its first $t_{j}$ binary digits from $z_{j}$, and an extra $\ell_{j+1}$ last digits. This is again guaranteed by enforcing:
\[
R_{j}(z_{j}, z_{j+1}) : z_{j} = \floor{z_{j+1} / 2^{\ell_{j+1}}} \quad \iff \quad
\bigg\{
\begin{matrix*}[l]
z_{j} \le z_{j+1} / 2^{\ell_{j+1}} \\
z_{j} > z_{j+1} / 2^{\ell_{j+1}} - 1
\end{matrix*}
\;\;.
\]
So now, if $R_{1}(z_{1},z_{2}),\dots,R_{k-2}(z_{k-2},z_{k-1})$ are all satisfied, then $z_{k-1}$ has $t_{k-1}$ digits corresponding to all digits from $y_{1},\dots,y_{k-1}$ concatenated.
If $\y_{k}$ has $n_{k}$ coordinates, we let $\z_{k}$ have $(k-1)+n_{k}$ coordinates. The last $n_k$ coordinates in $\z_{k}$ correspond to those in $\y_{k}$. The first $k-1$ coordinates in $\z_{k}$ are needed to recover $y_{1},\dots,y_{k-1}$ from $z_{k-1}$.
This is achieved by conditioning:
\begin{equation}\label{decompression}
z_{k-1} \. =  \. 2^{\ell_{2}+\dots+\ell_{k-1}} z_{k,1}  + 2^{\ell_{3}+ \. \dots \. + \ell_{k-1}} z_{k,2} + \dots
 \.\ldots \.  +  2^{\ell_{k-1}} z_{k,k-2} + z_{k,k-1},
\end{equation}
and
\begin{equation}\label{eq:z_bound}
\aligned
0 \le z_{k,1} < 2^{\ell_{1}} ,\; & \dots ,\; 0 \le z_{k,k-1}  < 2^{\ell_{k-1}} ,\; \\ 0 \le z_{k,k} \; , \, & \dots , \; z_{k,k-1+n_{k}} < 2^{\ell_{k}}.
\endaligned
\end{equation}

The whole system $\Psi(y_{1}, \dots, y_{k-1}, \y_{k})$ can now be expressed in $z_{k-1}$ and $\z_{k}$.
Indeed, we first rewrite the system $\Psi(y_{1}, \dots, y_{k-1}, \y_{k})$ with $$z_{k,1}\; ,\dots, \; z_{k,k-1}, \; z_{k,k}, \; \dots, \;z_{k,k-1+n_k}$$ in place of $$y_{1}, \; \dots, \; y_{k-1}, \;y_{k,1}, \; \dots, \; y_{k,n_k} \..$$
Now we let $\Lambda(z_{k-1}, \z_{k})$ be a new system including \eqref{decompression},~\eqref{eq:z_bound} and $\Psi$.
It is clear that $\Psi(y_{1}, \dots, y_{k-1}, \y_{k})$ holds if and only if $\Lambda(z_{k-1}, \z_{k})$ holds.
It is also clear that the new sentence $S''$ as in \eqref{type_1} or \eqref{type_2} has length $\polyin{\phi(S')}$ and is equivalent to the original sentence $S'$.
Note that $z_{1},\dots,z_{k-1}$ now have length bounds $t_{1} < \dots < t_{k-1}$, i.e., we require $0 \le z_{j} < 2^{t_{j}}$ for each of the first $k-1$ quantifier.
\end{proof}

Combining lemmas~\ref{red1}, \ref{red2}, and~\ref{red3}, we conclude that every sentence $S \in \Pbclass$ is equivalent to
a sentence $S''$ of the form~\eqref{type_1} or~\eqref{type_2} in some other class $\P^{\textup{b}}_{k,\n',a'}$.
The first $k-1$ variables in $S''$ are now singletons and the system $\Lambda(z_{k-1},\z_{k})$ involves only
the last two variables $z_{k-1}$ and $\z_{k}$. We say that such short Presburger sentences $S''$ are
in \emph{disassociated form}.

In order to prove Theorem~\ref{th:main_1}, we need Kannan's Partition Theorem.
Adopting the terminology in~\cite{K1}, we call a polyhedron with possibly some open facets a \emph{copolyhedron}.

\begin{theo}[Kannan's partition theorem]\label{KannanPartition}
Fix $n$ and $q$.
Consider a matrix $A \in \zz^{m \times n}$ of binary length $\phi$ and a $q$-dimensional polyhedron $W \subseteq \rr^{m}$.
For every $\b \in W$, let $K_{\b} = \{\x \in \R^{n} : A\x \le \b\}$. Assume that $K_{\b}$ is bounded for all $\b \in \R^{m}$. Then one can find in polynomial time a partition
\[
W \. = \. P_1 \sqcup P_2  \sqcup \dots \sqcup P_r,
\]
with $r \le (m n \phi)^{q n^{\delta n}}$, $\delta$ a universal constant, and each $P_{i}$ is a rational copolyhedron with the following properties.
For each $P_i$, $1\le i \le r$,  one can find in polynomial time a finite set $\T_i = \big\{ (T_{ij}, T'_{ij}) \big\}$ of pairs of rational affine transformations $T_{ij} : \R^{m} \to \R^{n}$ and $T'_{ij} : \Z^{n} \to \Z^{n}$, such that for every $\b \in P_i$, we have:
\begin{equation}\label{KannanCandidates}
K_{\b} \cap \Z^{n} \neq \varnothing \;\; \iff \;\; \ex (T_{ij}, T'_{ij}) \in \T_{i} \; : \; T'_{ij} \floor{ T_{ij} \b } \in K_{\b} \,.
\end{equation}
Furthermore, the size $\bigl|\T_{i}\bigr|\le n^{4n}$, for all $1\le i \le r$.
\end{theo}

\begin{rem}{\rm
If the number of rows $m$ in $A$ is fixed, each condition $T'_{ij} \floor{ T_{ij} \b } \in K_{\b}$ can be expressed as a short Boolean combination of linear inequalities, at the cost of introducing a few extra $\ex$ or $\for$ quantifiers. For example, the condition $\frac{1}{2} + \floor{b/5} \le 3$ for $b \in \R$ can be expressed as either
\begin{equation}\label{floor}
\exists\ts t
\begin{Bmatrix}\,
t &\le &b/5\\
t &> &b/5 - 1\\
\, \frac{1}{2} + t &\le &3\,
\end{Bmatrix}
\quad \text{or} \quad
\for\ts t
\begin{bmatrix} \,
t &> &b/5\\
t &\le &b/5 - 1\\
\, \frac{1}{2} + t &\le &3 \,
\end{bmatrix}.
\end{equation}
Here $\{\cdot\}$ is a conjuction and $\left[\cdot\right]$ is a disjunction.
}\end{rem}

\begin{theo}[Kannan]\label{t:Kannan}
Short sentences $\ts\for \ts \y \. \ex \ts \x \; \Phi(\x,\y)\ts$ in every fixed class $\P_{2,\n,a}$ can be decided in polynomial time.
\end{theo}

\begin{rem}{\rm
The idea of Theorem~\ref{t:Kannan}'s proof is to first partition the parameter space $\R^{n_{1}}$ for $\y$  into polynomially many copolyhedra using Theorem~\ref{KannanPartition}. For each copolyhedron, we have a finite set of candidates for $\x$, expressible using an extra quantifier $\for \t$ as in \eqref{floor}, which is then combined with the outer $\for \y$ quantifier. For the full proof, see~\cite{K1}. See also $\S$\ref{LenstraKannanShort} for a related remark.}
\end{rem}

\begin{proof}[Proof of Theorem~\ref{th:main_1}]
Consider a short disassociated Presburger sentence $S$ with variables $z_1,\dots,z_{k-1},\z_{k}$ of the form \eqref{type_1} or \eqref{type_2}. We induct on $k$, with the base case $k=2$ being Theorem~\ref{t:Kannan}. Now assume that for a fixed $k$ and every $\n', a'$, sentences in $\P_{k-1,\n',a'}$ are decidable in polynomial time. For convenience, we assume $k$ is odd; the case $k$ even is analogous. Then $S$ has the form:
\begin{equation}\label{disassociated}
\aligned
\exb z_1 \; \forb z_2 \; &  \lnot R_1(z_1, z_2) \lor \Bigl[ \exb z_3 \;  R_2(z_2,z_3) \land \bigl[ \. \ldots  \\
\dots &  \forb z_{k-1} \;  \lnot R_{k-2}(z_{k-2},z_{k-1}) \lor [ \exb \z_k \; \Lambda(z_{k-1}, \z_{k})  ] \dots  \bigr] \; \Bigr].
\endaligned
\end{equation}

Notice that the last system $\exb \z_{k} \, \Lambda(z_{k-1}, \z_{k})$ has fixed dimensions. If the system has $m$ inequalities, which is at most a constant, we can rewrite it  as
\[
\exb \z_{k} : A\z_{k} \le \albar z_{k-1} + \nubar \quad \text{with} \quad A \in \Z^{m \times n_{k}},\, \albar,\nubar \in \Z^{m}.
\]
For convenience, let $n = n_{k}$.
For each $z_{k-1}$, let $\b = \albar z_{k-1} + \nubar \in \rr^{m}$ and $$K_{z_{k-1}} \coloneqq \{\z_{k} \in \rr^{n} : \Lambda(z_{k-1},\z_{k})\} = \{\z_{k} \in \rr^{n} : A\z_{k} \le \b\}.$$
The set of all such $\b$ lies in a $1$-dimensional polyhedron in $W \subseteq \rr^{m}$.
We apply Theorem~\ref{KannanPartition} to the system $A\z_{k} \le \b$ with variables $\z_{k}$ and parameters $\b$. 
Theorem~\ref{KannanPartition} gives a polynomial size partition $W = P_{1} \sqcup \dots \sqcup P_{r}$, where $W$ is the set of all possible $\b$ as $z_{k-1}$ varies over $\rr$. 
This in turn induces a partition of $\R$, the parameter space for $z_{k-1}$, into
\begin{equation}\label{Rpartition}
\R = R_{1} \sqcup \dots \sqcup R_{r},
\end{equation}
where every $R_{i}$ is a rational interval.\footnote{Each $R_{i}$ can be half open with rational end points. Even though this forms a partition of $\R$, we only consider integer values in each $R_{i}$ for $z_{k-1}$.}
Since $\b = \cj\al z_{k-1} + \cj \nu$ depends affinely on $z_{k-1}$, by \eqref{KannanCandidates}, we have for each interval $R_i$ a constant size collection $\T_{i} = \{(T_{ij},T'_{ij})\}$ of pairs of rational affine maps $T_{ij} : \R \to \R^{n}$ and $T'_{ij} : \Z^{n} \to \Z^{n}$, so that for every $z_{k-1} \in R_{i}$ we have:
\begin{equation}\label{intermediate}
\begin{aligned}
\exb \z_{k} \, \Lambda(z_{k-1},\z_{k}) \; &\iff \; \ex (T_{ij}, T'_{ij}) \in \T_{i} \, : \, T'_{ij} \floor{ T_{ij} (z_{k-1}) } \in K_{z_{k-1}} \\
& \iff \, \bigvee_{j}  A \. T'_{ij} \floor{ T_{ij}(\albar z_{k-1} + \nubar) } \le \albar z_{k-1} + \nubar \\
& \iff \, \bigvee_{j} \for \t_j
\begin{bmatrix}
\t_j \neq \floor{ T_{ij}(\albar z_{k-1} + \nubar) } \\
~\\
A \. T'_{ij} \t_j \le \albar z_{k-1} + \nubar
\end{bmatrix}
,
\end{aligned}
\end{equation}
where the disjunction is over all $j$ such that $(T_{ij},T'_{ij}) \in \T_{i}$.

Here we are expressing the condition $\t_j \neq \floor{ T_{ij}(\albar z_{k-1} + \nubar) }$ using a short disjunction after $\for t$ as in \eqref{floor}. We have to do this for all coordinates $\. t_{j,1}, \. \dots \. ,t_{j,n} \ts$. The next step is to bring all the quantifiers $\forall \t_j$  outside of the short disjunction $\bigvee_{j}$ in \eqref{intermediate}. We can concatenate all $\t_{j}$'s into another vector~$\u$. Thus, for every $z_{k-1} \in R_{i}$, we have:
\begin{equation}\label{candidates}
\exb \z_{k} \; \Lambda(z_{k-1},\z_{k}) \quad \iff \quad \for \u \bigvee_{j}
\begin{bmatrix}
\u_{j} \neq \floor{ T_{ij}(\albar z_{k-1} + \nubar) } \\
~\\
A \. T'_{ij} \u_{j} \le \albar z_{k-1} + \nubar
\end{bmatrix}.
\end{equation}
Notice that $\u$ still has bounded dimension, because the number of pairs $(T_{ij},T'_{ij}) \in \T_{i}$ is at most $n^{4n}$. Also, the whole expression after $\for \u$ is still short.
\smallskip

Now comes the benefit of having $z_{1},\dots,z_{k-2}$ disassociated from $\Lambda(z_{k-1},\z_{k})$. Let us recall the proof of Lemma~\ref{red3}.
In there, the variables $z_{1},\dots,z_{k-1}$ have length bounds $t_{1} < \dots < t_{k-1}$.
For every $1 \le j \le k-2$, the relation $R_{j}(z_{j}, z_{j+1})$ forces $z_{j+1}$ to carry all the binary digits of $z_{j}$ as its first (most significant) $t_{j}$ binary digits. So if all $\. R_{1}(z_{1},z_{2}), \dots, R_{k-2}(z_{k-2},z_{k-1}) \.$ are all satisfied, then out of the $t_{k-1}$ digits of $z_{k-1}$, the first $t_{1}$ digits are from $z_{1}$. For particular value of $z_{1}$ in the range $[0,2^{t_{1}})$, every such $z_{k-1}$ lies in a contiguous segment of length $2^{t_{k-1}-t_{1}}$. To be precise, for every $z_{1} \in [0,2^{t_{1}})$, we have
\[
z_{k-1} \in I_{z_{1}} \coloneqq \bigl[ z_{1} 2^{t_{k-1}-t_{1}}, \; (z_{1}+1) 2^{t_{k-1}-t_{1}}  \bigr).
\]
There are $2^{t_{1}}$ such segments $I_{z_{1}}$, one for each $z_{1} \in [0,2^{t_{1}})$. However, by \eqref{Rpartition}, the domain $\R$ for $z_{k-1}$ was partitioned into $r$ (rational) segments $\. R_{1} \sqcup \dots \sqcup R_{r} \.$, where $r$ is polynomial in $\phi(S)$. Therefore, at most a polynomial number of intervals $I_{z_{1}}$ overlap with more than one interval $R_{i}$. We partition the interval $[0,2^{t_{1}})$ of all possible $z_{1}$ values into two subsets:
\begin{equation}\label{breakup}
\begin{aligned}
&\F_{1} = \big \{z_{1} \in [0,2^{t_{1}}) : I_{z_{1}} \subseteq R_{i} \text{ for some } 1 \le i \le r \big\} \quad \ \text{and} \\
 &\F_{2} = \big \{z_{1} \in [0,2^{t_{1}}) : I_{z_{1}} \text{ intersects both }  R_{i} \text{ and } R_{i+1} \  \text{ for some } \. i \big\}.
\end{aligned}
\end{equation}
In other words, $\F_{1}$ contains every interval $I_{z_{1}}$ that lies completely inside some interval $R_{i}$, and $\F_{2}$ contains the rest.
Observe that $|\F_{2}| \le r = \polyin{\phi(S)}$.
This is because the intervals $I_{z_{1}}$ are disjoint for different values of $z_{1}$, and if $z_{1} \in \F_{2}$ then $I_{z_{1}}$ must contain the common end point of $R_{i}$ and $R_{i+1}$ for some $1 \le i \le r$.

The original sentence $S$ begins with $\exb z_{1}$. First, we check over all values $z_{1} \in \F_{2}$. Substituting any such $z_{1}$ value into $S$, we get another short sentence with \emph{\textbf{one quantifier less}}, i.e., a sentence in some class $\P^{\textup{b}}_{k-1,\n',a'}$. By induction, each such sentence is polynomial time decidable. In summary, we can check whether any $z_{1} \in \F_{2}$ satisfies $S$, in time $\polyin{\phi(S)}$.

For $z_{1} \in \F_{1}$, recall by Theorem~\ref{KannanPartition} that one can find $R_{1},\dots,R_{r}$ in polynomial time. Thus, we can subpartition $\F_{1}$ into $r$ parts:
\begin{equation}\label{subinterval}
\F_{1} \. = \. \bigsqcup_{i=1}^{r}  \F_{1,i} \quad \text{where} \quad \F_{1,i} \.=\.  \bigl\{z_{1} \in \F_{1} : I_{z_{1}} \subseteq R_{i}\bigr\},\; 1\le i \le r.
\end{equation}
Note that each $\F_{1,i}$ is a contiguous subinterval in $[0,2^{t_{1}})$. For each $\F_{1,i}$, we can iteratively check if any $z_{1} \in \F_{1,i}$ satisfies $S$ as follows. For a fixed $i$ and all $z_{1} \in \F_{1,i}$, we have $z_{k-1} \in I_{z_{1}} \subseteq R_i$.
Therefore, by~\eqref{candidates}, the final quantifier $\exb \z_{k} \; \Lambda(z_{k-1},\z_{k})$ can be replaced by $\. \forb \u \; \Gamma_{i}(z_{k-1},\u).$
 Here $\Gamma_{i}$ as given by the RHS in \eqref{candidates} depends on $i$ but is still short. So now in \eqref{disassociated} we can combine $\forall z_{k-1}$ and $\forall \u$ together and get
\begin{equation}\label{reduce}
\aligned
\exb (z_{1} \in  \F_{1,i} ) \; \forb z_2 \; \lnot R_1(z_1, z_2) \lor \Bigl[ \exb z_3 \;  R_2(z_2,z_3) \land  \big[ \dots \\
 \dots  \forb z_{k-1} \forb \u \; \; \lnot R_{k-2}(z_{k-2},z_{k-1}) \;
 \lor \;  \Gamma_i(z_{k-1},\u) \dots \big] \; \Big].
\endaligned
\end{equation}
The quantifiers $\forb z_{k-1}$ and $\forb \u$ can be combined as $\forb (\z_{k-1},\u)$.
This results in a short sentence in some class $\lnot \P^{\textup{b}}_{k-1,\n'',a''}$ (negated because the last quantifier is $\forb$).
By the inductive assumption, we can check this sentence in polynomial time. In summary, we can check the sentence~\eqref{reduce} in polynomial time for each $1 \le i \le r$. Since $r$ is polynomial in $\phi(S)$, we can check the whole set $\F_{1}$ in time $\polyin{\phi(S)}$.
\smallskip

The case of even $k$ follows verbatim, with $\F_{2}$ consisting of subproblems in some class $\P^{\textup{b}}_{k-1,\n',a'}$ and $\F_{1}$ consisting of subproblems in some other class $\lnot\P^{\textup{b}}_{k-1,\n'',a''}$.
\end{proof}


\subsection{Finding short generating functions for short Presburger formulas}
A short Presburger formula is defined as a short Presburger sentence with the first variable $\x_{1}$ unquantified. We again group these formulas into families:

\begin{equation*}
\PFclass = \Big\{ \; F\. = \. \big[ \x_{1} \,:\, Q_2 \x_{2} \; Q_{3} \x_{3} \; \dots \; \exists \x_{k} \;\; \Phi(\x_{1}, \dots, \x_{k}) \big] \; \Big\}.
\end{equation*}
{}

\nin Here $k,\n,a$ have the same meanings as in \eqref{PclassDef}.
The $k-1$ quantifiers $Q_{2},\dots,Q_{k} \in \{\ex,\for\}$ alternate, with $Q_{k} = \ex$.
First, we prove a restricted version of Theorem~B:

\begin{theo}\label{th:main_2}
Assuming KPT, given a short formula $F \in \PFclass$ and a number $N$ in binary,
one can find a short \gf for
\[
\bigl\{\x_{1} \in \Z^{n_{1}} \cap [-N,N]^{n_{1}} \,:\, F(\x_{1}) = \textup{true}\bigr\}
\]
in time polynomial in $\phi(F)$ and $\log{N}$.
\end{theo}

As we mentioned in the introduction, the special case $k=2$ of the above theorem follows from Theorem~1.7 in~\cite{BW}
on projection of integer points in a finite dimensional polytope, which we restate below for convenience.

\begin{theo}[Barvinok and Woods]\label{BW}
Fix $m$. Given a rational polytope $P \subset \R^{m}$ described by $A\x \le \b$, and a linear transformation $T : \Z^{m} \to \Z^{n}$ represented by a matrix $T \in \zz^{n \times m}$, there is a polynomial time algorithm that computes a short GF for $T(P \cap \Z^{m})$ as:
\[
g(\t) \, = \, \sum_{\z \; \in \; T(P \cap \, \Z^{m})} \t^{\z}  \; = \; \sum_{i=1}^{M} \frac{c_{i} \. \t^{\a_{i}}}{(1-\t^{\b_{i1}}) \dots (1-\t^{\b_{i s}})}\,,
\]
where $c_{i}=p_{i}/q_{i} \in \Q ,\; \a_{i},\b_{ij} \in \Z^{n} ,\; \b_{ij} \neq 0$ for all $i,j$, and $s=s(m)$ is a constant depending only on~$m$.
\end{theo}


Define the \emph{length} of the short GF~$\ts g(\t)\ts$ as in Theorem~\ref{BW} as
\begin{equation}\label{GFsize}
\phi(g) \, = \,
\sum_{i} \. \lceil\log_2 |p_{i} \. q_{i}|+1\rceil \,  + \,
\sum_{i, j} \. \lceil\log_2 a_{i\ts j}+1\rceil +
 +  \sum_{i,j,r} \. \lceil\log_2 b_{i \ts j \ts r}+1\rceil\ts,
\end{equation}
where $\a_i=(a_{i\ts 1},\ldots,a_{i\ts n})$ and $\b_{i\ts j}=(b_{i\ts j\ts 1},\ldots,b_{i\ts j\ts n})$.

Referring back to the proof of Theorem~\ref{th:main_1}, we see that Theorem~\ref{th:main_2} can be proved following the same vein if we assume $n_{1} = 1$, i.e., $\x_{1}$ is a singleton $x_{1}$. If $n_{1} > 1$, we can first convert $\x_{1}$ into a singleton by concatenating its (bounded) coordinates into a single number $x_{1}$ as in Lemma~\ref{red2}. The cases corresponding to positive and negative coordinates $x_{1,j}$ can be treated separately. However, doing so would affect the multi-variable generating function for $\x_{1}$. The following technical result is a GF analogue of Lemma~\ref{red2}, which allows one to convert between multi-variable and single-variable short generating functions.

\begin{lem}\label{1dimgf}
Fix $n$. Assume $F \subseteq [0,2^{\ell})^{n}$ has a short \gf $f(\t)$ which expands into $\sum_{\x \in F} \t^{\x}$.
Let $G\subseteq \bigl[0,2^{n\ell}\bigr)$ be defined as
\[
G \ts := \ts \bigl\{x_{1} + 2^{\ell}x_{2} + \dots + 2^{(n-1)\ell} x_{n} : (x_{1}, \dots, x_{n}) \in F \bigr\} \..
\]
Then $G$ has a short \gf $g(t)$ of length $\ts \polyin{\phi(f) + \ell}$ which expands into $\sum_{x \in G} t^{x}$.
Conversely, if $G$ has a short \gf $g(t)$, then $F$ also has a short \gf $f(\t)$ of length
$\ts\polyin{\phi(g) + \ell}$.
\end{lem}

\begin{proof}[Proof of Lemma~\ref{1dimgf}]
Let $N = 2^{\ell}$. Assume the formula $F$ has a short \gf $f(\t)$ that satisfies
\[
f(\t) = \sum_{\x \in F} \t^{\x} = \sum_{\x \in F} t_{1}^{x_{1}} \dots \ts t_{n}^{x_{n}}.
\]
Let $g(t)$ be the evaluation of $f(\t)$ under the following substitutions:
\[
t_{1} \gets t, \; t_{2} \gets t^{N}, \ldots, \; t_{n} \gets t^{N^{n-1}},
\]
so that
\[
\t^{\x} \. = \. t^{x_{1} + N x_{2} + \ldots + N^{n-1} x_{n-1}}.
\]
Clearly, GF $g(t)$ expands into $\sum_{x \in G} t^{x}$.  Thus it is a
short generating function for $G$. By Theorem~2.6 in \cite{BW},
the above monomial substitutions on $f(\t)$ can be performed in
polynomial time, giving $g(t)$ of polynomial length.
\smallskip

For the other direction, assume $G$ has a short \gf $g(t)$.
%
Consider the following multi-variable short \gf $a(\t)$:
\[
a(\t) \. = \, \sum_{\x \in [0,N)^{n} } \t^{\x} \, = \, \frac{1-t_{1}^{N}}{1-t_{1}} \. \cdots \.\frac{1-t_{n}^{N}}{1-t_{n}}\..
\]
Since $n$ is fixed, after expanding product in the numerators, we have $a(\t)$ a short \gf of length $\polyin{\log N}$.

Define a linear map $\tau : \Z^{n} \to \Z$ as:
\[
\tau(\x) \. = \. x_{1} + N x_{2} + \ldots + N^{n-1} x_{n}.
\]
Given $A(\t) = \sum \al_{\x} \t^{\x}$ a multi-variable short GF and $B(t) = \sum \be_{x} t^{x}$ a single-variable short GF, we define their \emph{$\tau$-Hadamard product} $C(\t) = A(\t) \tauHad B(t)$ as follows:
\begin{equation}\label{tauHadDef}
A(\t)  \tauHad B(t) \coloneqq \sum \al_{\x} \be_{\tau(\x)} \t^{\x}\..
\end{equation}
From this definition, it is clear that our original set $F \in [0,N)^{n}$ has a GF given by:
\[
f(\t) = a(\t) \tauHad g(t).
\]
We prove the following claim:
The $\tau$-Hadamard product of two short GFs is again a short GF of polynomial length.
The proof is an analogue of Barvinok's argument in \cite{B2} (see also lemmas 3.4 and~3.6 in \cite{BW}). First, notice that the $\tau$-Hadamard product is bilinear in $A(\t)$ and $B(t)$. Therefore, it suffices to prove the claim when $A(\t)$ and $B(\t)$ each has only one term, i.e.,
\begin{equation}\label{singleterm}
A(\t) \, = \, \frac{\t^{\a}}{\prod_{i=1}^{p} (1 - \t^{\b_{i}})} \ \quad \text{and} \quad B(t) \, = \, \frac{t^{c}}{\prod_{j=1}^{q} (1-t^{d_{j}})}\..
\end{equation}

Consider an (unbounded) polyhedron $\ts P \subset \R^{p + q}\ts$ with coordinates $\ts (\zeta_{1}, \dots, \zeta_{p}, \xi_{1}, \dots, \xi_{q})$, defined as:
\begin{equation}\label{HadamardPolytope}
P \coloneqq
\begin{Bmatrix}
\zeta_{1},\dots,\zeta_{p},\xi_{1},\dots,\xi_{q} & \ge &0 \\
\tau(\a + \zeta_{1} \b_{1} + \dots + \zeta_{p} \b_{p}) & = &c + \xi_{1} d_{1} + \dots + \xi_{q} d_{q}
\end{Bmatrix}.
\end{equation}
By Theorem~2.2 from \cite{B1}, we can write a short GF for $P \cap \Z^{p+q}$:
\begin{equation}\label{diagonal}
D(\u,\v) \coloneqq \sum_{(\zzeta,\xxi) \in P} \u^{\zzeta} \v^{\xxi} = \sum_{(\zzeta,\xxi) \in P} u_{1}^{\zeta_{1}} \dots u_{p}^{\zeta_{p}} \, v_{1}^{\xi_{1}} \dots v_{q}^{\xi_{q}}.
\end{equation}
By \eqref{singleterm}, the expansions of $A(\t)$ and $B(t)$ are:
\begin{equation}\label{expansion}
A(\t) = \sum_{\zzeta \ge 0} \t^{\a + \zeta_{1} \b_{1} + \dots + \zeta_{p} \b_{p}}, \  B(t) = \sum_{\xxi \ge 0} t^{c + \xi_{1} d_{1} + \dots + \xi_{q} d_{q}}.
\end{equation}
We substitute
\[
u_{1} \gets \t^{\b_{1}} , \dots , u_{p} \gets \t^{\b_{p}} \quad \text{and} \quad v_{1} \gets 1 , \dots , v_{q} \gets 1.
\]
By \eqref{HadamardPolytope}, \eqref{diagonal} and ~\eqref{expansion}, we get

\[
\t^{\a} D(\t^{b_{1}}, \dots , \t^{b_{p}}, 1, \dots , 1) = A(\t) \tauHad B(t) = C(\t).
\]
Since substitutions can be done in polynomial time, we obtain a short GF $C(\t)$ of polynomial length.
This completes the proof.
\end{proof}

\smallskip

\begin{proof}[Proof of Theorem~\ref{th:main_2}]
First, we make a change of variables from $\x_{1}$ to $\x'_{1} = \x_{1} + N$, i.e., $x'_{1,j} = x_{1,j} + N$.
So counting the number of $\x_{1} \in [-N,N]^{n_{1}}$ is equivalent to counting the number of $\x'_{1} \in [0,2N]^{n_{1}}$.
Therefore, we can assume that all coordinates of $\x_{1}$ are non-negative.

Given a formula in $\PFclass$, we can apply Lemmas~\ref{red1}, \ref{red2} and \ref{red3} to convert it into an equivalent formula $F$ in disassociated form as in \eqref{disassociated} (with $\exb z_{1}$ replaced by ``$z_{1} : $''). The vector $\x_{1}$ is now a singleton $z_{1}$ bounded in some interval $[0, 2^{t_{1}})$.
Applying Lemma~\ref{1dimgf}, it is equivalent to show that the GF $\. f(t) = \sum_{z_{1}} t^{z_{1}} \.$ is short. We prove the result by induction on $k$. The case $k=2$ follows from Theorem~\ref{BW}.

Assume that for fixed $k$ and all $\n'$ and $a'$, every formula in $\PF_{k-1,\n',a'}$ has a short GF of polynomial length in every finite interval $[0,N)$. Applying the same reasoning as in the proof of Theorem~\ref{th:main_1}, we get a partition for $[0,2^{t_{1}})$ into $\F_{1}$ and $\F_{2}$, see~\eqref{breakup}. Recall that $|\F_{2}|$ is polynomial in $\phi(F)$. Substituting each value $z \in \F_2$ into $F$ for $z_{1}$, we get a fully quantified short Presburger statement $S_{z}$ in some class $\P^{\bounded}_{k-1,\n',a'}$~, with $\phi(S_{z}) = \polyin{\phi(F)}$. Each such statement $S_{z}$ can be checked in time $\polyin{\phi(S_{z})}$ by Theorem~\ref{th:main_1}. Therefore, in time $\polyin{\phi(F)}$, we obtain a short GF $g(t)$:
\[
g(t) = \sum_{z \in \F_{2} \; : \; S_{z} = \text{true}} t^{z}.
\]

By \eqref{subinterval}, we have a refinement of $\F_{1}$ into polynomially many intervals $\F_{1,i}$, where $1 \le i \le r$. By \eqref{reduce}, for $z_{1} \in \F_{1,i}$, the formula $F$ is equivalent to another formula $F_i$ in some class $\lnot\PF_{k-1,\n'',a''}$, with $\phi(F_{i}) = \polyin{\phi(F)}$. The GF $f_{i}(t)$ for $F_i$ can be found in time $\polyin{\phi(F_{i})}$ by induction.

In summary, we obtain in time $\polyin{\phi(F)}$, the GF
\[
f(t) \. = \. \sum_{i=1}^{r} \. f_{i}(t) + g(t),
\]
which completes the proof.
\end{proof}

We can actually remove the coordinate bounds in Theorem~\ref{th:main_2}:

\begin{theo}\label{th:main_2_strong}
Assuming KPT, given a short formula $F \in \PFclass$, we can find a short \gf for
\[
\bigl\{\x_{1} \in \Z^{n_{1}} \,:\, F(\x_{1}) = \textup{true}\bigr\}
\]
in time polynomial in $\phi(F)$.
\end{theo}

\begin{proof}
By Theorem 5.3 in~\cite{NP}, given a Presburger formula $F$, the full generating
function $f(\t)$ for all satisfying $\x_{1}$ can be computed in polynomial time
given a partial generating function $f_{N}(\t)$ for satisfying $\x_{1}$ in a
large enough box $[-N,N]^{n_{1}}$.
This result also allows us to compute $N$ in polynomial time given $F$.
With such an $N$, we can appeal to Theorem~\ref{th:main_2} to compute $f_{N}(\t)$ so that $\phi(f_{N})$ is polynomial in $\log N$ and $\phi(F)$.
Since $\log N = \polyin{\phi(F)}$, we also have $\phi(f_{N}) = \polyin{\phi(F)}$.
By an application of Theorem 5.3 in~\cite{NP}, we recover the full generating function $f$,
which satisfies $\phi(f) = \polyin{\phi(f_{N})} = \polyin{\phi(F)}$.\footnote{Generally
speaking, one needs to be careful taking evaluations and Hadamard products
for bi-infinite Laurent power series, to avoid summations of the type $\sum_{n\in \zz} t^n$.
Paper~\cite{NP} sidesteps this problem by explicitly disallowing such summations.
We refer to~\cite{B3,BP} for the theory of valuations in this context, which allows
one to get around this issue.
}
\end{proof}

\begin{rem}
Here we treat the full generating function of $\x_{1}$ as formal power series which can also be represented as a rational function $f(\t)$.
In some cases, the power series might not converge under numerical substitution.
For example, if $F$ is a trivial formula then every $\x_{1} \in \Z^{n_{1}}$ satisfies $F$.
So the power series for $\x_{1}$ is $\sum_{\x_{1} \in \zz^{n_{1}}} \t^{\x_{1}}$, which is not convergent for any non-zero $\t$.
However, if $\x_{1}$ is restricted to lie in a pointed cone, for example $\x_{1} \in \nn^{n_{1}}$, then the power series converges on a non-empty open domain.
For any $\t$ in that domain, the power series converges to the computed rational function $f(\t)$.

\end{rem}

\bigskip

\section{Final remarks}

\subsection{Long systems}\label{LenstraKannanShort}
Recall that both Lenstra and Kannan's results on deciding sentences of
types $(\circ)$ and $(\circ\circ)$ as in the introduction allow for \emph{long systems of inequalities}.
However, we can reduce each case to deciding a polynomial numbers of short sentences.
Indeed, let $n$ be fixed and $m\ge 2^n$ be arbitrary.  The \emph{Doignon--Bell--Scarf theorem}
\cite[$\S$16.5]{Schrijver} (see also~\cite{ABDL}) implies that a system $A\x \le \b$ with $A \in \zz^{m \times n}$ has an integer solution $\x \in \Z^{n}$ if and only if every
short subsystem $A' \x \le \cj{b'}$ has a solution $\x \in \Z^{n}$.
Here $A'$ is a submatrix with $2^{n}$ rows from $A$,
and $\cj{b'}$ is the corresponding subvector from~$\b$.

For one quantifier $\ex$, by the Doignon--Bell--Scarf theorem, we have:
\[
\ex\ts \x \,:\, A\x \le \b \quad \iff \quad \bigwedge_{(A',\cj{b'})} \ex\ts \x \,:\, A'\x \le \cj{b'} \ts.
\]
So it is equivalent to decide each of the ${m}\choose{2^{n}}$ short sentences individually.
This number clearly polynomial in $m$ if~$n$ is fixed.

For two quantifiers $\for\ts\ex$, in the system $A'\x + B' \y \le \cj{c'}$ we can proceed
in a similar manner, see~\cite[$\S$7.1]{NP1}.
However, already for three quantifiers as in $(\circ\mts\mts{}\circ\mts\mts{}\circ)$
this approach provably fails.
Roughly, this is because the long conjunction over $(A',B',\cj{c'})$
no longer commutes with the outer existential quantifier $\ex\ts \z \in R$.

In fact, our most recent result~\cite{NP1} proves that for long systems as
in~$(\circ\mts\mts{}\circ\mts\mts{}\circ)$, the problems becomes $\NP$-complete,
already for $\ov n = (1,2,3)$.  This negatively resolves an open problem in~\cite{K2}
and underscores the contrast with Theorem~A.

\medskip

\subsection{Bounded affine dimension}\label{ss:finrem-Eisenbrand-Shmonin}
In \cite{ES}, Eisenbrand and Shmonin strengthened Kannan's Partition Theorem (Theorem~\ref{KannanPartition})
by completely removing the condition that the parameter space $W \subset \rr^{m}$ has a bounded affine dimension~$q$.
However, in their final result (\cite[Th.~4.1]{ES}), the parameter space~$W$ is partitioned
into $Q_{1} \sqcup \dots \sqcup Q_{r}$,
where each $Q_{i} \subset \R^{m}$ is no longer a copolyhedron.  Instead, each $Q_{i}$
is now the \emph{integer projection} of some higher dimensional
rational copolyhedron $Q'_{i} \subset \R^{m+k}$, defined as:
 \[
 Q_{i} \coloneqq \{\b \in \R^{m} \,:\, \ex\ts \y \in \Z^{k} \;\; (\b,\y) \in Q'_{i}\}.
 \]
Here $k$ is a constant that depends only on $n$.

 Note that having each piece $P_{i}$ as an
 actual copolyhedron (interval for $m=1$), is crucial for our proof of
 Theorem~\ref{th:main_1}.  For this, see the partition into intervals $R_i$
 in~\eqref{Rpartition}, and a discussion that follows.

\medskip

\subsection{Validity of Kannan's Partition Theorem}\label{ss:finrem-Kannan}

The proof of KPT given in~\cite{K1} is quite technical and relies on
an earlier conference paper which was later revised and published separately~\cite{K2},
which in turn uses the \emph{flatness theorem} (as did~\cite{BW,ES}), and other earlier work.
While we have no doubt in the validity of Kannan's Theorem~\ref{t:Kannan},
in part due to its self-contained presentation and generalization in~\cite{ES}
(see also~\cite{E2}), we were unable to piece together all the details which
go into the proof of KPT.
However, at this time we are not ready to establish a clear gap in the proof
of KPT, which would revert its status to a conjecture.  We are simply being
cautious in citing a theorem whose proof we do not fully understand, and
which is crucially used as a black box in the proof of both theorems~A and~B.

In the near future, we intend to bring more clarity into validity of~KPT,
at least in the $m=1$ case which is used in the paper.  In the meantime
we intend to treat KPT as an oracle, a time honored tradition in both
computational logic and computational complexity.  We hope this clarifies
the reasoning behind our somewhat nonstandard use of KPT as an assumption
in the statements of the results.

\vskip.6cm


\subsection*{Acknowledgements}
We are greatly indebted to Sasha Barvinok for his eternal optimism, many fruitful discussions and encouragement.
We are also thankful to Iskander Aliev, Matthias Aschenbrenner, Art\"{e}m Chernikov,
Jes\'{u}s De Loera, Lenny Fukshansky, Oleg Karpenkov and Sinai Robins for interesting
conversations and helpful remarks.  We are very grateful to Rafi Ostrovsky and Vijay Vazirani
for kindly advising us on the structure of the paper, and to anonymous referees for
their comments.  The second author was partially supported by the~NSF.



\end{document}